\documentclass[11pt]{amsart}
\usepackage{amsmath}
\usepackage{amssymb}
\usepackage{pxfonts}
\usepackage{amsfonts}
\usepackage{mathrsfs}
\usepackage{wasysym}
\usepackage{enumitem}

\theoremstyle{plain}
\newtheorem{thm}{Theorem}[section]
\newtheorem{mth}{Theorem}

\newtheorem{lem}[thm]{Lemma}

\theoremstyle{remark}

\newtheorem*{rem}{Remark}

\numberwithin{equation}{section}

\newcommand{\sample}{\{\zeta_j\}_1^n}
\newcommand{\propco}{s}
\newcommand{\param}{h_*}
\newcommand{\an}{a_0}
\newcommand{\pa}{h_0}

\newcommand{\Qeff}{Q^{\,\mathrm{eff}}}

\newcommand{\bfR}{\mathbf{R}}

\newcommand{\calF}{{\mathcal F}}

\newcommand{\calW}{{\mathscr W}}

\newcommand{\bfE}{{\mathbf E}}

\newcommand{\bfP}{{\mathbf P}}

\newcommand{\Int}{\operatorname{Int}}

\newcommand{\R}{{\mathbb R}}

\newcommand{\C}{{\mathbb C}}

\newcommand{\fii}{{\varphi}}

\newcommand{\eps}{{\varepsilon}}

\newcommand{\Ham}{H}

\renewcommand{\d}{{\partial}}
\newcommand{\dbar}{\bar{\partial}}
\newcommand{\1}{\mathbf{1}}

\newcommand{\dist}{\operatorname{dist}}
\newcommand{\supp}{\operatorname{supp}}

\newcommand{\Lap}{\Delta}

\def\ie{\textit{i.e.}}
\def\viz{\textit{viz.}}

\begin{document}



\title[Localizing the Coulomb gas]{A localization theorem for the planar Coulomb gas in an external field}

\begin{abstract} We examine a two-dimensional Coulomb gas consisting of $n$ identical repelling point charges at an arbitrary inverse temperature $\beta$, subjected to a suitable external field.

We prove that the gas is effectively localized to a small neighbourhood of the droplet -- the support of the equilibrium measure determined by the external field. More precisely, we prove that the distance between the droplet and the vacuum is with very high probability at most proportional to $$\sqrt{\dfrac {\log n}{\beta n}}.$$ This order of magnitude is known to be ``tight'' when $\beta=1$ and the external field is radially symmetric.

In addition, we prove estimates for the one-point function in a neighbourhood of the droplet, proving in particular a fast uniform decay as one moves
beyond a distance roughly of the order $\sqrt{\tfrac{\log n}{\beta n}}$ from the droplet.
\end{abstract}

\author{Yacin Ameur}

\address{Yacin Ameur\\
Department of Mathematics\\
Faculty of Science\\
Lund University\\
P.O. BOX 118\\
221 00 Lund\\
Sweden}

\email{Yacin.Ameur@math.lu.se}

\maketitle

\section{Introduction and Main Results} \label{mainr}



 The planar Coulomb gas is a random configuration consisting of many (but finitely many) identical repelling point charges $\{\zeta_i\}_1^n$ in $\C$.

 To keep the system from dispersing to infinity we assume the presence of an external field $nQ$ where $Q$ is a suitable
extended real-valued function defined on $\C$, large near infinity in the sense that
\begin{equation}\label{gro}\liminf_{\zeta\to\infty}\frac {Q(\zeta)}{2\log|\zeta|}>1.\end{equation}

 The function $Q$, which is called an external potential, is fairly general but not quite arbitrary; precise
assumptions are given below.

\smallskip

To a planar configuration $\{\zeta_j\}_1^n$ we associate the Hamiltonian (or energy)
\begin{equation}\label{hami}H_n=\sum_{j\ne k}^n\log\frac 1 {|\zeta_j-\zeta_k|}+n\sum_{j=1}^n Q(\zeta_j).\end{equation}

\smallskip

The statistical model is completed by fixing an inverse temperature $\beta=1/k_B T$ and considering $\{\zeta_j\}_1^n$ as a random sample with respect to the Gibbs measure 
\begin{equation}\label{gibb}d\bfP_n^\beta=\frac 1 {Z_n^\beta}\, e^{-\beta H_n}\,dA_n.\end{equation}

Here and throughout we write $dA$ for the Lebesgue measure in $\C$ divided by $\pi$ and
$dA_n(\zeta_1,\ldots,\zeta_n)=dA(\zeta_1)\cdots dA(\zeta_n)$ for the corresponding product measure, where
the constant $Z_n^\beta$ in \eqref{gibb} is the usual partition function,
\begin{equation}\label{pf}Z_n^\beta=\int_{\C^n} e^{-\beta H_n}\, dA_n.
\end{equation}

\smallskip

We shall now gradually become explicit about our precise assumptions, introducing simultaneously our basic objects of study.

\begin{enumerate}[label=(\roman*)]
\item\label{A0}
 The external potential $Q$ is assumed to be a fixed lower semi-continuous function on $\C$ with values in
$\R\cup\{+\infty\}$ such that the set $\Sigma:=\{Q<+\infty\}$ has non-empty interior $\Int\Sigma$. We also suppose that $Q$ is $C^2$-smooth in $\Int\Sigma$,
and that $Q$ obeys the growth condition \eqref{gro}.
\end{enumerate}



We next define the \textit{logarithmic $Q$-energy} of a finite, compactly supported Borel measure $\mu$ on $\C$ by
$$I_Q[\mu]=\iint_{\C^2}\log\frac {1} {|\zeta-\eta|}\,d\mu(\zeta)d\mu(\eta)+\mu(Q),$$
where ``$\mu(Q)$'' is shorthand for $\int_\C Q\, d\mu$.

It is shown in \cite{ST} that
there is a unique \textit{equilibrium measure} $\sigma$ of total mass $1$, which minimizes $I_Q[\mu]$ over all compactly supported Borel probability measures $\mu$. The support of $\sigma$, which we denote by the symbol
$$S=S[Q]:=\supp\sigma$$ is called the \textit{droplet} in external potential $Q$; we stress that this is a compact set. It is convenient to make the following assumption.

\begin{enumerate}[label=(\roman*)]
\setcounter{enumi}{1}
\item\label{A} We assume that the droplet $S=S[Q]$ is contained in $\Int\Sigma$ where $\Sigma=\{Q<+\infty\}$.
\end{enumerate}

By \ref{A0} and \ref{A} it follows that $Q$ is smooth in a neighbourhood of $S$. It is well-known (see \cite{ST}) that in this circumstance, the equilibrium measure $\sigma$ is absolutely continuous with respect to $dA$ and takes the form
\begin{equation}\label{cont}d\sigma=\Lap Q\cdot \1_S\, dA.\end{equation}
Here and in what follows we normalize the Laplacian by
$$\Lap Q:=\d\dbar Q=\frac 1 4(Q_{xx}+Q_{yy}),\qquad (z=x+iy).$$

Observe that since $\sigma$ is a probability measure, $Q$ is necessarily subharmonic on the support $S$.

\medskip

We now list additional conditions which we always assume to be satisfied in the sequel.

\begin{enumerate}[label=(\roman*)]
\setcounter{enumi}{2}
\item\label{B} The potential $Q$ is strictly subharmonic in a neighbourhood of the boundary $\d S$.
\item\label{C} The boundary $\d S$ has finitely many components.
\item\label{D} Each component of $\d S$ is an everywhere $C^1$-smooth Jordan curve.
\item\label{E} $S^*=S$ where $S^*$ is the coincidence set for the obstacle problem, given in Section \ref{Sec2}.
\end{enumerate}

\smallskip

Some of these conditions are assumed merely for convenience and may be relaxed. The practically minded reader may note that the class of potentials which are real-analytic where finite typically works well, and sufficiently small smooth perturbations of such potentials are also manageable. A detailed discussion of this situation is given in Subsection \ref{pertpot}.

\medskip

We will denote the Euclidean distance between a point $\zeta\in\C$ and the compact set $S$ by the symbol
$$\delta(\zeta)=\dist(\zeta,S):=\min\{|\zeta-\eta|;\, \eta\in S\}.$$

Let $\{\zeta_j\}_1^n$ be a random sample from \eqref{gibb}.
If $W$ is a measurable subset of $\C$ we will denote by $\#(W\cap\sample)$ the number of indices $j$ such that $\zeta_j\in W$. By definition, the
\textit{one-point function} of $\{\zeta_j\}_1^n$ is
$$\bfR_n^\beta(\zeta)=\lim_{\eps\to 0}\frac {\bfE_n^\beta[\# (D(\zeta;\eps)\cap\sample)]}{\eps^2},$$
where $D(\zeta;\eps)$ is the open disc $\{\eta;\, |\zeta-\eta|<\eps\}$.

We next define a positive constant $c_0=c_0[Q]$ by
\begin{equation}\label{cdef}c_0:=\min\{\Lap Q(\eta);\,\eta\in \d S\}.\end{equation}

Given these preliminaries, we have the following theorem.

\begin{mth}\label{DTH} Let $Q$ be a potential satisfying assumptions \ref{A0}-\ref{E}. There exists an open neighbourhood $\Omega$ of $S$ and constant $C>1$ such that for all $n\ge 1$
\begin{equation}\label{decay}\bfR_n^\beta(\zeta)\le C^\beta n^2 e^{-c\beta n\cdot \delta(\zeta)^2},\qquad \zeta\in \Omega.\end{equation}
Here $c$ can be taken to be any positive constant with $c<c_0$.
\end{mth}

The exterior decay given in \eqref{decay} comes close to known exact results for $\beta=1$ as well as projected (or ``physical'') results for arbitrary $\beta$ in
\cite{CFTW}.

We note that the uniform bound $\bfR_n^{\,\beta}(\zeta)\le C^\beta n^2$ of Theorem \ref{DTH} can be improved in the determinantal case $\beta=1$ to $\bfR_n^{\,1}\le Cn$ (e.g. \cite[Section 3]{AKM}). We are not aware of similar uniform bounds for
$\beta\ne 1$.

\smallskip

Associated to a random sample $\{\zeta_j\}_1^n$ we now define the number
\begin{equation}\label{Dn}D_n:=\max_{1\le j\le n}\{\delta(\zeta_j)\}.\end{equation}
We regard $D_n$ as a random variable with respect to the Gibbs measure, which thus represents the \textit{distance from the droplet to the vacuum}.

\begin{mth}\label{TH1} Assume that the external potential $Q$ satisfies assumptions \ref{A0}-\ref{E}.

Let $\beta=\beta_n$ be a possibly $n$-dependent inverse temperature which is not too small in the sense that
\begin{equation}\label{ass0}\lim_{n\to\infty}\frac {
\beta_n n}{\log n}
=\infty.\end{equation}
Then there exists
 a sequence $\mu=\mu_n$ of positive numbers
with \begin{equation}\label{tess}\mu_n\lesssim \log\log n+\beta_n,\end{equation} and constants $c>0$ and $a>0$, such that for each real $t$ satisfying $$t\le a\beta_n n$$ we have the estimate
\begin{equation}\label{jes}\bfP_n^\beta\left(\left\{D_n>\sqrt{\frac {\log n+\mu_n+t} {c\beta_n n}}\right\}\right)\le e^{-2t}.\end{equation}
Here $c$ and $a$ depend only on $Q$; $c$ can be taken as any positive number with $c<c_0$, where $c_0$ is given in \eqref{cdef}.
\end{mth}

\begin{rem} The meaning of the notation $a_n\lesssim b_n$ is that there exists a number $n_0$ and a constant $C$ such that
$a_n\le Cb_n$ for all $n\ge n_0$. (The number $n_0$ may depend on the sequence $(\beta_n)$ and $Q$, while $C$ depends only on $n_0$ and $Q$.)
The symbol $a_n\asymp b_n$ means
$a_n\lesssim b_n$ and $b_n\lesssim a_n$.
\end{rem}

Recall from \cite{HM} that (for any fixed $\beta>0$) the system $\{\zeta_j\}_1^n$ tends to follow the equilibrium measure $\sigma$ in the sense that
\begin{equation}\label{eqc}\lim_{n\to\infty}\frac 1 n\bfE_n^\beta(f(\zeta_1)+\cdots+f(\zeta_n))\to \sigma(f)\end{equation}
for each bounded continuous function $f$.
This 
implies, in a loose sense, that that the particles are likely to stay in the immediate vicinity of the droplet. It could be said that Theorem \ref{TH1} gives more detailed information about exactly how ``localized'' the gas is about the droplet.

\smallskip

To illustrate this point, we may observe that if we fix a $\beta>0$ and
choose $t=t_n$ so that $t_n\to\infty$ and $t_n/\log n\to 0$ as $n\to\infty$,
then \eqref{jes} implies
\begin{align}\label{loca}A>\sqrt{\frac 1 {c_0}}\qquad \text{implies}\qquad
  \lim_{n\to\infty}\bfP_n^\beta&\left(\left\{D_n>A\sqrt{\frac {\log n}{\beta n}}\right\}\right)=0.\end{align}

 Hence if $A$ satisfies the premise in \eqref{loca}, then the gas is effectively localized to the set of $\zeta$ with \begin{equation}\label{smova}\delta(\zeta)<A\sqrt{\frac {\log n} {\beta n}}.\end{equation}

The estimate \eqref{smova} might be compared with earlier results on the distribution of the spectral radius
of certain types of normal random matrices, due to Rider \cite{R} for the Ginibre ensemble, cf.~\cite{CP}
for more general ensembles corresponding
to radially symmetric potentials $Q$.

Indeed, as is well-known, we can interpret the Coulomb gas $\{\zeta_j\}_{j=1}^n$ in external potential $Q$ at inverse temperature $\beta=1$ as eigenvalues of normal random matrices. (See \cite{CZ,EF} for details.) Let us temporarily assume that $Q$ is radially symmetric and that
the droplet is the disc centered at $0$ of radius $R$.

A normal matrix with eigenvalues $\{\zeta_j\}_1^n$ has its spectral radius equal to $\max_{1\le j\le n}|\zeta_j|$, and if $D_n>0$
then clearly
\begin{equation}\label{iden}R+D_n=\max_{1\le j\le n}|\zeta_j|.\end{equation}

The identity \eqref{iden} may of course fail if $D_n=0$, but since $\bfP_n^1(\{D_n=0\})\to 0$ rapidly as $n\to\infty$ by estimates in \cite{CP,R}, we may regard \eqref{iden} as ``essentially'' an identity.

 We shall show in Subsection \ref{comp}
that the estimate \eqref{smova} comes close to earlier results on spectral radii 
in \cite{R,CP}, 
in the sense that the order of magnitude of our obtained localization is comparable with what is obtained in those papers.

\smallskip

While our main focus is on the case when $t=t_n$ increases slowly to infinity in the sense that $t_n/\beta n$ is ``small'', it is also
relevant to note the following theorem, concerning the case when $t_n/\beta n$ is ``large''.

\begin{mth}\label{mprop} Keeping the conditions in Theorem \ref{TH1}, there exists a number $r_0>0$ such that for each $r\ge r_0$ there is a number $k(r)>0$ such that
$$\bfP_n^{\beta_n}(\{D_n>r\})\lesssim \exp (-k(r)\cdot \beta_n n).$$
Indeed, we may take $k(r)$ proportional to $\min\{\Qeff(\zeta);\, \delta(\zeta)\ge r\}$ where $\Qeff$ is the effective potential defined in Section \ref{Sec2}.
\end{mth}

We remark that Theorem \ref{mprop} is closely related to a result of Chafa\"{i}, Hardy, and Ma\"{i}da, which holds in dimension $d\ge 2$. See \cite[Theorem 1.12]{CHM}.

\subsection*{Plan of this paper} In Section \ref{Sec2} we give some background on potential theory and weighted polynomials.
Theorem \ref{DTH} is proven in Section \ref{Sec3} while
 Section \ref{Sec_proof}
 contains proofs of Theorem \ref{TH1} and Theorem \ref{mprop}.
In Section \ref{concrem} we will state and prove generalized versions of the above theorems, and we also discuss some related earlier work in the area.

\section{Preparation}\label{Sec2}

In order to make this note as detailed and complete as possible, we shall now review some notions from the theories of obstacle problems and of weighted polynomials. We shall also discuss, in a suitably adapted form, some relevant background from \cite{A}.

As general sources for some statements taken for granted below, we refer to the book \cite{ST} and the paper \cite{HM}.

We stress, once and for all, that in the following, the external potential $Q$ is assumed to be a fixed function satisfying assumptions \ref{A0}-\ref{E} above.

\subsection{The obstacle problem} \label{obs}
Let $\calF_Q$ be the family of all subharmonic functions $f$ on $\C$ which are everywhere $\le Q$ and which satisfy $f(\zeta)\le 2\log|\zeta|+O(1)$ as $\zeta\to\infty$.

We define a subharmonic function $\check{Q}$ on $\C$ by
$$\check{Q}(\zeta)=\sup\{f(\zeta);\, f\in\calF_Q\}.$$
This is the \textit{obstacle function} corresponding to the obstacle $Q$; it is well-known and easy to check that $\check{Q}$ satisfies
$\check{Q}\le Q$ and $\check{Q}(\zeta)=2\log|\zeta|+O(1)$ as $\zeta\to\infty$.

\smallskip

Borrowing notation from \cite{BBLM}, we define the \textit{effective potential}
$$\Qeff:=Q-\check{Q},$$
and note that $\Qeff\ge 0$ on $\C$.
By the \textit{coincidence set} for the obstacle problem we shall mean the compact set
$$S^*:=\{\Qeff=0\}.$$

It is well-known (cf. \cite{HM}) that
$\check{Q}$ is $C^{1,1}$-smooth on $\C$ and harmonic in the complement $(S^*)^c$. (``$C^{1,1}$-smooth'' means that the gradient is Lipschitz continuous.)
Moreover,
$\check{Q}$ is related to the equilibrium measure $\sigma$ by
$$\check{Q}(\zeta)=-2U^\sigma(\zeta)+\gamma$$
where $\gamma$ is a suitable (Robin's) constant and
$$U^\sigma(\zeta):=\int_\C\log\frac {1} {|\zeta-\eta|} \, d\sigma(\eta)$$ is the logarithmic potential of $\sigma$.

Differentiating in the sense of distributions we have
$$\Lap \check{Q}=\Lap Q\cdot \1_{S^*}=-2\Lap U_\sigma=\sigma.$$
Hence, since $S$ is the support of $\sigma$, we have the inclusion
$$S \subset S^*.$$

In general the difference set $S^*\setminus S$ may be non-empty, consisting then of ``shallow points'' in the parlance of \cite{HM}. However, our
assumption \ref{E} says precisely that there are no shallow points, \ie, that $S^*=S$. Thus condition \ref{E} can be restated as that
\begin{equation}\label{po}\Qeff>0\qquad \text{on}\qquad S^c.\end{equation}
(Here and in the following, $S^c:=\C\setminus S$.)

\smallskip

We will have frequent use for the following simple lemma.
A proof is included for completeness.

\begin{lem} \label{lem1} There exists a number $\an>0$ such that
$$\Qeff(\zeta)\ge 2\min\left\{c\cdot \delta(\zeta)^2,\an\right\}\qquad \text{for\, all}\qquad \zeta\in S^c.$$
Here the constant $c$ is any positive number with $c<c_0$ (cf. \eqref{cdef}).
\end{lem}

\begin{proof} Fix a boundary point $p\in \d S$ and let $N=N_p$ be the unit normal to $\d S$ at $p$ pointing outwards from $S$.
Let $V$ be a $C^2$-smooth (Whitney's) extension of $\Qeff|_{S^c}$ to a neighbourhood of $\d S$.

We shall write $\d_N V(p)$ for the directional derivative
in direction $N$ and $\d_T V(p)$ for the derivative in the (positively oriented) tangential direction to $\d S$.

By the $C^{1,1}$-smoothness of $\Qeff$ and the fact that $\Qeff=0$ on $S$ we have
$V(p)=\d_NV(p)=0$. Moreover, $(\d_{N}^2+\d^2_T)V(p)=4\Lap Q(p)>0$ and $\d^2_TV(p)=0$.

For small $\delta>0$ we hence obtain by Taylor's formula that
$$\Qeff(p+\delta N)=V(p+\delta N)=2\Lap Q(p)\delta^2+o(\delta^2),\qquad (\delta\to 0+).$$

  Moreover, there is $\delta_0>0$ such that $\Qeff(\zeta)\ge  2c\delta(\zeta)^2$ when $\delta(\zeta)\le\delta_0$.

Finally, by the lower semi-continuity of
$\Qeff$ and the assumptions \eqref{gro}, \eqref{po} we conclude that $\Qeff$ attains a strictly positive minimum over the set $\{\delta(\zeta)\ge \delta_0\}$. The lemma follows if we denote this minimum value by ``$2a_0$''.
\end{proof}

We next note the following lemma -- a simple consequence of our growth assumption on the external potential $Q$.

\begin{lem}\label{LemF} $\int_\C Qe^{-Q}\, dA<\infty$ where $Qe^{-Q}:=0$ on $\{Q=+\infty\}$.
\end{lem}

\begin{proof} Let $\alpha$ be a number in the range $1<\alpha<\liminf_{\zeta\to\infty}Q(\zeta)/\log|\zeta|^2$, cf. \eqref{gro}.
Fix a number $\theta<1$ such that $\theta\alpha >1$ and let $x_\theta$ be a real number such that
$xe^{-x}\le e^{-\theta x}$ when $x\ge x_\theta$. Choosing $x_\theta$ somewhat larger if necessary we can also assume that
$Q(\zeta)\ge \alpha \log|\zeta|^2$ when $|\zeta|\ge x_\theta$. Then $Q(\zeta)e^{-Q(\zeta)}\le |\zeta|^{-2\alpha\theta}$ when
$|\zeta|\ge x_\theta$, proving the lemma.
\end{proof}

\subsection{Weighted polynomials}
Let $\calW_n$ denote the subspace of $L^2=L^2(\C,dA)$ consisting of elements (weighted polynomials)
$$f=q\cdot e^{-nQ/2}$$
where $q$ is a holomorphic polynomial of degree at most $n-1$.

\smallskip

The following well-known lemma is sometimes known
as the ``maximum principle of weighted potential theory''. We outline a proof for convenience.

\begin{lem} \label{lem2} If $f\in\calW_n$ and $\zeta\in\C$ then
$$|f(\zeta)|\le \|f\|_{L^\infty(S)}\cdot e^{-n\Qeff(\zeta)/2}.$$
\end{lem}

\begin{proof} We may assume that $\|f\|_{L^\infty(S)}=1$ where $f=q\cdot e^{-nQ/2}$.
Since
$$\frac 1 n \log|q(\zeta)|^2=\frac 1 n\log|f(\zeta)|^2+Q(\zeta)$$ we see that the function $u:=\frac 1 n\log|q|^2$, which is subharmonic on $\C$, satisfies
$u\le Q$ on $S$, and furthermore $u(\zeta)\le \log|\zeta|^2+O(1)$ as $\zeta\to\infty$.

A suitable version of
the maximum principle now shows that $u\le\check{Q}$ on $\C$, thus finishing the proof of the lemma.
\end{proof}

 We now fix, once and for all, an open neighbourhood $V$ of the droplet $S$, which is small enough so that $\Lap Q$ is continuous and strictly positive in a neighbourhood $V_1$ of the closure $\overline{V}$. This is possible by assumption \ref{B}.

\smallskip


Next fix a number $\propco$ with
$$\propco>\max\{\Lap Q(\zeta);\, \zeta\in \overline{V}\}.$$

\begin{lem} \label{pl2} There exists $n_0>0$ such that if $n\ge n_0$ then for each $f\in\calW_n$ and each $\zeta_0\in \overline{V}$ we have the pointwise-$L^{2\beta}$ estimate
$$|f(\zeta_0)|^{2\beta}\le ne^{\,\propco \beta}\int_{D(\zeta_0;1/\sqrt{n})}|f|^{2\beta}\, dA.$$
\end{lem}

\begin{proof} Consider the function
$$F(\zeta)=|f(\zeta)|^{2\beta}e^{\propco n\beta |\zeta|^2}.$$
Writing $f=q\cdot e^{-nQ/2}$ we have
$$\Lap \log F(\zeta)\ge -\beta n\Lap Q(\zeta)+\propco\beta n\ge 0$$
for all $\zeta\in V_1$. 
We infer that $F$ is logarithmically subharmonic
in $V_1$, and in particular it is subharmonic there.

Applying the sub mean-value inequality, we obtain (for $n$ large enough that $1/\sqrt{n}\le\dist(V,V_1^c)$) the estimate
\begin{align*}F(\zeta_0)&\le n\int_{D(\zeta_0;1/\sqrt{n})} F\, dA\le ne^{\,\propco \beta}\int_{D(\zeta_0;1/\sqrt{n})}|f|^{2\beta}\, dA.
\end{align*}
The proof is complete.
\end{proof}

\subsection{Random variables} Let $\{\zeta_j\}_1^n$ be a random sample with respect to the Gibbs measure \eqref{gibb}.

Consider now, for a fixed $j$ with $j\in\{1,\ldots,n\}$ the random \textit{weighted Lagrange polynomial}
\begin{equation}\label{lag}\ell_j(\zeta)=\left(\prod_{i\ne j}(\zeta-\zeta_i)/\prod_{i\ne j}(\zeta_j-\zeta_i)\right)\cdot e^{-n(Q(\zeta)-Q(\zeta_j))/2}.\end{equation}
These weighted polynomials were used in \cite{A} to study the separation of random configurations. We shall here use similar techniques to examine the localization of the gas.

Towards this end, let us fix a measurable subset $W\subset \C$ and write
$Y_j=Y_{j,W}$ for the random variable
\begin{equation}\label{rvs}Y_{j,W}:=\int_{W}|\ell_j(\zeta)|^{2\beta}\, dA(\zeta).\end{equation}

\subsection{Exact identities}
The following lemma is a slight generalization of \cite[Lemma 2]{A}; it
will play a key r\^{o}le in what follows.


\begin{lem}\label{pl3}
Let $U\subset\C$ be a measurable subset, of $dA$-measure $|U|$. Then
\begin{equation*}\bfE_n^\beta\left[\1_U(\zeta_j)\cdot Y_{j,W}\right]=|U|\cdot p_n^\beta(W),\qquad (j=1,\ldots,n),\end{equation*}
where
$$p_n^\beta(W):=\bfP_n^\beta(\{\zeta_j\in W\})=\bfP_n^\beta(\{\zeta_1\in W\}).$$
\end{lem}

\begin{proof}
We start with the basic identity
\begin{equation*}\label{rem}|\ell_j(\zeta)|^{2\beta}e^{-\beta H_n(\zeta_1,\ldots,\zeta_j,\ldots,\zeta_n)}=e^{-\beta H_n(\zeta_1,\ldots,\zeta,\ldots,\zeta_n)}.\end{equation*}
By Fubini's theorem, integrating first in $\zeta_j$, we get
\begin{align*}
\int_W &dA(\zeta)\,\bfE_n^{\beta}\left[|\ell_j(\zeta)|^{2\beta}\cdot\1_{U}(\zeta_j)\right]\\
&=\int_{U}dA(\zeta_j)\int_{\zeta\in W,\,\zeta_k\in\C,\,(k\ne j)}d\bfP_n^{\beta}(\zeta_1,\ldots,\zeta,\ldots,\zeta_n)
=|U|\cdot p_n^\beta(W).\end{align*}
The proof is complete.
\end{proof}

The argument above may be iterated. To illuminate the principle, we start by considering the case of two different indices $j$ and $k$, $1\le j<k\le n$.

\begin{lem}\label{pl4} If $j<k$  and $U_1,U_2,W_1,W_2$ are any measurable subsets of $\C$ then
\begin{align*}\bfE_n^\beta \left[\1_{U_1}(\zeta_j)\cdot Y_{j,W_1}\cdot \1_{U_2}(\zeta_k)\cdot Y_{k,W_2} 
\right]
=|U_1||U_2|\cdot p_{n,2}^\beta(W_1,W_2),
\end{align*}
where
$$p_{n,2}^\beta(W_1,W_2):=\bfP_n^\beta(\{\zeta_1\in W_1,\,\zeta_2\in W_2\}).$$
\end{lem}

\begin{proof} Note that
\begin{equation*}|\ell_j(\zeta)|^{2\beta}|\ell_k(\eta)|^{2\beta}e^{-\beta H_n(\zeta_1,\ldots,\zeta_j,\ldots,\zeta_k,\ldots,\zeta_n)}=e^{-\beta H_n(\zeta_1,\ldots,\zeta,\ldots,\eta,\ldots,\zeta_n)}.\end{equation*}
with $\zeta$ and $\eta$ in positions $j$ and $k$, respectively. The lemma follows by using Fubini's theorem as in the preceding proof.
\end{proof}

More generally, if $\{j_1,\ldots,j_m\}$ is any subset of $\{1,\ldots,n\}$
with $j_1<j_2<\cdots<j_m$ then with a self-explanatory notation
\begin{equation}\label{inter}\begin{split}\bfE_n^\beta&\left[\1_{U_1}(\zeta_{j_1})\cdot Y_{j_1,W_1}\cdot \1_{U_2}(\zeta_{j_2})\cdot Y_{j_2,W_2}\cdots \1_{U_m}(\zeta_{j_m})\cdot Y_{j_m,W_m}\right]
\\
&\qquad\qquad\qquad\qquad=|U_1||U_2|\cdots |U_m|\cdot p_{n,m}^\beta(W_1,\ldots W_m).\\
\end{split}
\end{equation}
The proof of this formula is straightforward and is omitted.

In the following we will consider the case when all the sets $W_j$ coincide with a set $W$, in which case we will use the abbreviation
$$p_{n,m}^\beta(W):=p_{n,m}^\beta(W,W,\ldots,W).$$
(``$W$'' occurs $m$ times in the right hand side.)


\section{Proof of Theorem \ref{DTH}} \label{Sec3}

To begin, we fix a positive number $c<c_0$ and a small enough open neighbourhood $\Omega$ of $S$ such that $c\delta(\zeta)^2<a_0$ for each $\zeta\in \Omega$.
(Here $c_0$ is fixed as in \eqref{cdef} and $a_0$ is the constant from Lemma \ref{lem1}.) We assume also that $\Omega$ is small enough that Lemma \ref{pl2} applies with $V=\Omega$.

Now fix a point $\zeta_0\in\Omega$ and a number $\eps>0$ which is small enough that the neighbourhood $W=D(\zeta_0;\eps)$ is contained in $\Omega$.

Next define a non-negative number $\delta$ by
$$\delta:=\dist(S,W)=\inf\{|\zeta-\eta|;\, \zeta\in S,\, \eta\in W\}.$$
Thus $\delta\to\delta(\zeta_0)$ as $\eps\to 0$.

For $1\le j\le n$ we consider the random variables
\begin{align*}Y_j&=\int_W|\ell_j|^{2\beta}\, dA,\\
 Z_j&=\int_{\C}|\ell_j|^{2\beta}\, dA.
 \end{align*}

By Lemma \ref{pl2} we have the global estimate
$$|\ell_j(\zeta)|^{2\beta}\le ne^{\,\propco\beta}\int_{D(\zeta,1/\sqrt{n})}|\ell_j|^{2\beta}\le ne^{\,\propco\beta}Z_j,\qquad \zeta\in\Omega.$$
In particular this holds for all $\zeta\in S$, so since $\ell_j\in\calW_n$ we obtain from lemmas \ref{lem2} and \ref{lem1} that
$$|\ell_j(\zeta)|^{2\beta}\le ne^{\,\propco\beta}e^{-cn\beta \delta^2}Z_j,\qquad \zeta\in W.$$

Integrating the latter inequality over $W$ using that $|W|=\eps^2$, we find that
\begin{equation}\label{chug}Y_j\le ne^{\,\propco\beta}e^{-cn\beta \delta^2}\eps^2Z_j.\end{equation}

Now fix a measurable subset $U$ of positive, finite measure and recall from Lemma \ref{pl3} that
\begin{align*}\bfE_n^\beta\left[\1_U(\zeta_j)\cdot Y_j\right]&=|U|\cdot p_n^\beta(W),\\
 \bfE_n^\beta\left[\1_U(\zeta_j)\cdot Z_j\right]&=|U|\cdot p_n^\beta(\C)=|U|.
 \end{align*}

If we multiply through in \eqref{chug} by $\1_U(\zeta_j)$ and then take expectations, we obtain the inequality
\begin{equation}\label{werke}p_n^\beta(W)\le 
ne^{\,\propco\beta}e^{-cn\beta \delta^2}\eps^2.\end{equation}

Next define for $j\ne k$
\begin{align*}Y_{j,k}&=\int_{W^2} |\ell_j(\zeta)\ell_k(\eta)|^{2\beta}\, dA_2(\zeta,\eta),\\
Z_{j,k}&=\int_{\C^2} |\ell_j(\zeta)\ell_k(\eta)|^{2\beta}\, dA_2(\zeta,\eta).
\end{align*}
Applying the argument above, we find first that
$$|\ell_j(\zeta)\ell_k(\eta)|^{2\beta}\le (ne^{\,\propco\beta}e^{-cn\beta \delta^2})^2Z_{j,k},\qquad \zeta,\eta\in W,$$
and then, by integrating over $W^2$,
$$Y_{j,k}\le (ne^{\,\propco\beta}e^{-cn\beta \delta^2})^2Z_{j,k}\eps^4.$$
Taking expectations in this inequality and using Lemma \ref{pl4} it now follows that
$$p_{n,2}^\beta(W)\le (ne^{\,\propco\beta}e^{-cn\beta \delta^2})^2p_{n,2}^\beta(\C)\eps^4=(ne^{\,\propco\beta}e^{-cn\beta \delta^2})^2\eps^4.$$

By the same token, using the identity \eqref{inter}, we obtain for each $k$, $1\le k\le n$ that
$$p_{n,k}^\beta(W)\le q^kp_{n,k}^\beta(\C)\eps^{2k}=q^k\eps^{2k}\qquad \text{where}\qquad
q=ne^{\,\propco\beta}e^{-cn\beta \delta^2}.$$
From this, we infer that
\begin{align*}&\bfE_n^\beta(\#(W\cap\sample))=\bfP_n^\beta(\{\#(W\cap\sample)=1\})\\
&\quad+2\bfP_n^\beta(\{\#(W\cap\sample)=2\})+\cdots+n\bfP_n^\beta(\{\#(W\cap\sample)=n\})\\
&\quad \le {n \choose  1}\,q\,\eps^2+2\,{n\choose 2}\,q^2\,\eps^4+\cdots+n\,{n\choose n}\, q^n\,\eps^{2n}.
\end{align*}

We finally conclude that
\begin{align*}\bfR_n^\beta(\zeta_0)=\lim_{\eps\to 0}\frac {\bfE_n^\beta(\#(W\cap\sample))}{\eps^2}\le n^2e^{\,\propco\beta}e^{-cn\beta \delta(\zeta_0)^2},\end{align*}
as desired. Q.E.D.


\section{Proofs of Theorem \ref{TH1} and Theorem \ref{mprop}} \label{Sec_proof}

We start by fixing a sequence $(U_n)_{n=1}^\infty$ of bounded open neighbourhoods of the droplet $S$, which is increasing and exhausts $\C$, \viz
$$S\subset U_1\subset U_2\subset\cdots,\qquad  \bigcup_{n=1}^\infty U_n=\C.$$

In the sequel we fix an unspecified integer $n_0$ which can be chosen larger as we go along, and we assume that $n\ge n_0$.

For definiteness, we will fix $U_n$ to be the disc of radius $\log n$ about the origin,
\begin{equation}\label{undef}U_n:=D(0;R_n),\qquad  R_n:=\log n,\qquad (n\ge n_0).\end{equation}

Now fix $j$, $1\le j\le n$, and put
$$X_j=\int_\C |\ell_j|^{2\beta}\, dA.$$
By Lemma \ref{pl3} we know that
\begin{equation}\label{zoo}\bfE_n^\beta\left[\1_{U_n}(\zeta_j)\cdot X_j\right]=|U_n|=R_n^2.\end{equation}

Now introduce the two events
\begin{equation}\label{events}A_j:=\{\zeta_j\in U_n\},\qquad\qquad  B_j:=\{X_j\le \lambda\},\end{equation}
where $\lambda>0$ is a parameter.

By Chebyshev's inequality and \eqref{zoo} we have the basic estimate
\begin{align*}\bfP_n^\beta(A_j\cap B_j^c)&\le \frac {R_n^2}{\lambda}.\end{align*}
Passing to complements we conclude that
\begin{equation}\label{tv}\bfP_n^\beta(A_j^c)+\bfP_n^\beta(B_j)\ge 1-\frac {R_n^2}{\lambda}.\end{equation}

We shall now prove that the probability $\bfP_n^\beta(A_j^c)$ is ``negligible''.

\begin{lem} \label{lem3}
There are constants $\param>0$ and $n_0>0$ such that $n\ge n_0$ implies
$$\bfP_n^\beta(A_j^c)\le R_n^{-\beta \param n},\qquad 1\le j\le n.$$
\end{lem}

\begin{proof} We will give a proof based on estimates for the partition function which can essentially be found in the union of the papers \cite{HM,J}. (\cite{J} is written in the setting of a real $\log$-gas, but the following argument is virtually the same in the complex case.)

\smallskip

We first note, due to the growth assumption \eqref{gro} on $Q$, that there are numbers $\pa>0$ and $n_0>0$ such that
for all $\zeta,\eta\in\C$ with $|\eta|>R_{n_0}$ we have
\begin{equation}\label{nyeq}\log \frac1 {|\zeta-\eta|^2}+Q(\zeta)+Q(\eta)>\pa\log|\eta|.\end{equation}
(To see this, use the elementary inequality $|\zeta-\eta|^2\le(1+|\zeta|^2)(1+|\eta|^2)$.)

\smallskip

Recalling the definition of the partition function $Z_n^\beta$ in \eqref{pf}, we now
write
\begin{align*}\bfP_n^\beta &(A_j^c)=\bfP_n^\beta(A_n^c)=\bfP_n^\beta(\{\zeta_n\in U_n^c\})=
\frac {Z_{n-1}^\beta}{Z_n^\beta}\int_{U_n^c}e^{-n\beta Q(\zeta_n)}dA(\zeta_n)\times\\
&\times\int_{\C^{n-1}}\exp\left(-\beta\sum_{j=1}^{n-1}\left[\log\frac 1 {|\zeta_j-\zeta_n|^2}+Q(\zeta_j)\right]\right)\, d\bfP_{n-1}^\beta(\zeta_1,\ldots,\zeta_{n-1}).
\end{align*}

Using \eqref{nyeq} we now obtain
\begin{align*}\bfP_n^\beta (A_j^c)&\le
\frac {Z_{n-1}^\beta}{Z_n^\beta}\int_{U_n^c}e^{-\beta Q(\zeta_n)}e^{-\beta(n-1)\pa \log|\zeta_n|}\, dA(\zeta_n)\\
&\le \frac {Z_{n-1}^\beta}{Z_n^\beta} R_n^{-\beta \pa(n-1)/2}\int_{U_n^c}e^{-\beta Q(\zeta_n)}|\zeta_n|^{-\beta\pa(n-1)/2}\, dA(\zeta_n).
\end{align*}

We conclude that there are constants $C$ and $h_1>0$ such that
\begin{equation}\label{lower}\bfP_n^\beta (A_j^c)\le C\frac {Z_{n-1}^\beta}{Z_n^\beta}R_n^{-\beta h_1 n}.\end{equation}

We shall thus be done when we can prove an upper bound of the form
\begin{equation}\label{joh}\frac {Z_{n-1}^\beta}{Z_n^\beta}\le  Ce^{\,\beta h_2n}.\end{equation}

For fixed $\beta$ this is shown in Section 5 of \cite{CHM} (see eq. (5.4)), and it also
 follows formally from a well-known large $n$ expansion of the partition function in \cite{ZW}. As we want to
 ensure a certain uniformity in $\beta$, we shall here
give an elementary proof based on the papers \cite{HM,J}.

We start with the identity
\begin{equation}\label{ko}\frac {Z_n^\beta}
{Z_{n-1}^\beta}=\bfE_{n-1}^\beta\left[\int_\C\exp\left(\beta\sum_{j=1}^{n-1}\log|t-\zeta_j|^2-n\beta Q(t)\right)\, dA(t)\right].\end{equation}

Now write $I:=\int_\C e^{-Q}$ and use Jensen's inequality to conclude that the last expression is
\begin{align}\label{op}
\ge I\cdot \exp\left\{\bfE_{n-1}^\beta\left[\int_\C\left(\beta\sum_{j=1}^{n-1}\log|t-\zeta_j|^2-(\beta n-1)Q(t)\right)\, \frac {e^{-Q(t)}} I \,dA(t)\right]\right\}.
\end{align}

We must estimate this expression from below. For this, we start by estimating the number
\begin{align*}m_n:&=\frac 1 n \bfE_{n-1}^\beta\left[\int_\C\sum_{j=1}^{n-1}\log|t-\zeta_j|^2\, e^{-Q(t)}\, dA(t)\right]\\
&\ge \frac 1 n \bfE_{n-1}^\beta\left[\int_\C\sum_{j=1}^{n-1}\log_-|t-\zeta_j|^2\, e^{-Q(t)}\, dA(t)\right]\end{align*}
where $\log_-x=\min\{\log x,0\}$.

Now fix a number $\delta$, $0<\delta<1$ so that $\int_{|t-\zeta|<\delta}\log |t-\zeta|^2\, e^{-Q(t)}\, dA(t)>-1$ for each $\zeta\in\C$.
Then
\begin{align*}m_n&> -1+\frac 1 n \bfE_{n-1}^\beta\left[\sum_{j=1}^{n-1}\int_{|t-\zeta_j|\ge\delta}\log_-|t-\zeta_j|^2\, e^{-Q(t)}\, dA(t)\right]\\
&>-1+2\log\delta+\frac 1 n \bfE_{n-1}^\beta\left[\int_{\C}\sum_{j=1}^{n-1}l_\delta(t-\zeta_j)\, e^{-Q(t)}\, dA(t)\right],
\end{align*}
where $l_\delta(\zeta):=2\max\{\log_- |\zeta|,\log\delta\}$.

Now as $l_\delta$ is bounded and continuous on $\C$
we can apply the convergence in
\eqref{eqc} to obtain \begin{equation}\label{into}\frac 1 n\bfE_{n-1}^\beta(l_\delta(t-\zeta_1)+\cdots+l_\delta(t-\zeta_{n-1}))\to \int_\C l_\delta(t-\zeta)\, d\sigma(\zeta)\end{equation} for each $t\in\C$.

In fact, an examination of the proof of the convergence \eqref{eqc} in \cite{HM,J} holds uniformly in $t$ provided that $\beta=\beta_n$ satisfies $\beta_nn\to\infty$ 
as
$n\to\infty$.
For convenience of the reader, we have collected the relevant details in the appendix; see in particular Theorem \ref{joh1}.


Integrating both sides of the limit \eqref{into} with respect to the measure $e^{-Q(t)}\, dA(t)$ we infer that the sequence $m_n$ is uniformly bounded from below when $\beta \gg \frac {\log n} n$, which is all that we need to know here.

In view of Lemma \ref{LemF}, we now obtain by \eqref{ko} and \eqref{op} that
a bound of the form \eqref{joh} must hold.

The proof of the lemma is complete.
\end{proof}

\begin{rem} If we slightly strengthen our assumptions on the potential, then
the above uniform convergence in \eqref{into} also follows from a quantitative result in \cite[Theorem 1.5]{CHM}. (The function $l_\delta$ belongs to the bounded Lipschitz class used there.)
\end{rem}

\subsection{Proof of Theorem \ref{TH1}}
Now fix an $n\ge n_0$ and recall that $B_j=\{X_j\le\lambda\}$ where $\lambda>0$ is fixed. (See \eqref{events}.)

It follows from Lemma \ref{lem3} and the inequality \eqref{tv} that
$$\bfP_n^\beta(B_j)\ge 1-R_n^2 \lambda^{-1}-R_n^{-\beta \param n},\qquad j=1,\ldots, n.$$

In order to be able to neglect the last term, we shall impose the following (mild) condition on the parameter $\lambda$,
\begin{equation}\label{paradrop}\lambda\lesssim R_n^{\beta\param n/2}.
\end{equation}

Under this assumption we have
$$\bfP_n^\beta(B_j)\ge 1-KR_n^2 \lambda^{-1}$$
for some constant $K$.
Hence if we
write
$$B:=\bigcap_{j=1}^n B_j,$$
then
\begin{align}\label{skans}\bfP_n^\beta(B)\ge 1-\epsilon,
\end{align}
where we have put
\begin{align}\label{epsdef}\epsilon:=KnR_n^2\lambda^{-1}.\end{align}

\smallskip

In the sequel we assume that $\zeta$ belongs to the fixed
neighbourhood $V$ of $S$. Pick a random sample $\{\zeta_j\}_1^n$.

\smallskip

By Lemma \ref{pl2} we have that
\begin{align*}|\ell_j|^{2\beta}(\zeta)&\le ne^{\,\propco \beta} X_j,\qquad
(X_j=\int|\ell_j|^{2\beta}).\end{align*}

Thus assuming that $B_j$ has occurred, we obtain
\begin{equation}\label{laban}\|\ell_j\|_{L^\infty(S)}^{2\beta}\le ne^{\,\propco\beta}\lambda.
\end{equation}

We now apply
Lemma \ref{lem2} to conclude that
$$|\ell_j(\zeta)|^{2\beta}\le ne^{\,\propco\beta}\lambda \cdot e^{-n\beta \Qeff(\zeta)},\qquad (\zeta\in\C).$$

In view of Lemma \ref{lem1}
we deduce that if $\zeta\in S^c$ then
\begin{equation}\label{Gugge}|\ell_j(\zeta)|^{2\beta}\le ne^{\,\propco \beta}\lambda
\exp\left(-2n\beta \min\left\{c\delta(\zeta)^2,\an\right\}\right).\end{equation}

\smallskip

To proceed, we now impose the further restriction on $\lambda$ that
\begin{equation}\label{resl}\lambda\lesssim e^{n\beta \an}.\end{equation}
We also fix an arbitrary number $\alpha\in(0,1)$ and choose $n_0$ large enough that $n\ge n_0$ implies
\begin{equation}\label{restr}ne^{\,\propco\beta}\lambda e^{-2n\beta \an}<\alpha.\end{equation}
This is possible by \eqref{resl} since \eqref{ass0} implies that $n\beta a_0> \log n$ for large $n$.

\smallskip

For a fixed $n\ge n_0$ we now
consider the set $M_n\subset S^c$ of points $\zeta$ such that
\begin{equation}\label{2}ne^{\,\propco \beta}\lambda e^{-2c\beta n\delta(\zeta)^2}\le \alpha.\end{equation}
Note that if $B_j$ has occurred then certainly $\zeta_j\not \in M_n$, for otherwise \eqref{Gugge} would imply $$\alpha\ge |\ell_j(\zeta_j)|^{2\beta}=1>\alpha.$$
Thus with probability
at least $1-\epsilon$ the entire Coulomb gas is actually contained in the neighbourhood $M_n^c$ of $S$.

\smallskip

To finish the proof we observe that our restrictions on $\lambda$ are equivalent to that the parameter $\epsilon$ (cf. \eqref{epsdef}) satisfy
\begin{equation}\label{nyenskans}\epsilon\gtrsim nR_n^2 e^{-n\beta \an}\qquad \text{and}\qquad \epsilon\gtrsim
nR_n^2 e^{-\beta\param n/2}.\end{equation}
For such $\epsilon$ we now write
$$\lambda=\frac {n\nu_n}\epsilon,\qquad \text{where}\qquad \nu_n=KR_n^2+o(1).$$
The inequality \eqref{2} is then equivalent to that
$$n^2e^{\,\propco\beta}\frac {\nu_n} \epsilon \le \alpha e^{2c\beta n\delta^2},\qquad (\delta=\delta(\zeta)),$$
which transforms to
$$\delta^2\ge \frac 1 {2c\beta n}\left(\log \frac 1 \alpha+\propco\beta+2\log n+\log \frac {\nu_n}\epsilon\right).$$
Thus if we define
$$\delta_n=\sqrt{\frac 1 {2c\beta n}\left(\log \frac 1 \alpha+\propco\beta+ 2\log n+\log \frac {\nu_n}\epsilon\right)}$$
then certainly
$$M_n\subset\{\zeta\in\C;\, \delta(\zeta)\ge \delta_n\}.$$
so the Coulomb gas is with probability at least $1-\epsilon$ contained in the $\delta_n$-neighbourhood of $S$. In symbols, we have shown that
$$\bfP_n^\beta\left(\left\{D_n\ge \delta_n\right\}\right)\le \epsilon.$$

If we write $\epsilon=e^{-t}$ and $\mu_n=\log\nu_n+\log(1/\alpha)+\propco\beta$ this becomes
\begin{equation}\label{fjun}\bfP_n^\beta\left(\left\{D_n>\sqrt{\frac 1 {c\beta n}\left(\log n+\mu_n/2+t/2\right)}\right\}\right)\le e^{-t}.\end{equation}

Now notice that $\log\nu_n\asymp \log\log n$ as $n\to\infty$, so $\mu_n\asymp \log\log n+\beta$.

We have shown \eqref{fjun} under the hypothesis \eqref{nyenskans}. Since we have assumed that $\beta\cdot n\to\infty$ as $n\to\infty$, \eqref{nyenskans}
surely holds if $t\le a\beta n$ for a small enough $a>0$.

The proof of Theorem \ref{TH1} is complete. Q.E.D.

\subsection{Proof of Theorem \ref{mprop}} It is now easy to modify the above proof so as to also prove Theorem \ref{mprop}.

Fix a number $r_0$ with $cr_0^2\ge \an$
and suppose that $\delta(\zeta)\ge r\ge r_0$. The estimate \eqref{Gugge} then takes the form
\begin{equation}\label{fjutt}|\ell_j(\zeta)|^{2\beta}\le ne^{\,\propco\beta}\lambda e^{-2n\beta k_0},\end{equation}
where $k_0$ is some constant with $k_0\ge \an$.
A glance at the proof of Lemma \ref{lem1} shows that we may take $2k_0=\min\{\Qeff(\zeta);\, \delta(\zeta)\ge r\}$.


For a fixed $\alpha\in (0,1)$ we now choose $n_0$ large enough that $n\ge n_0$ implies
$$ne^{\,\propco\beta}\lambda e^{-2n\beta k_0}\le \alpha.$$
This can certainly be done if we assume that $\lambda\lesssim e^{n\beta k_0}$, in view of \eqref{ass0}.

If the event $B_j=\{X_j\le \lambda\}$ has occurred then $\delta(\zeta_j)< r$, for otherwise $\alpha\ge |\ell_j(\zeta_j)|^{2\beta}=1>\alpha$.

Thus if we define, as before, $\epsilon=KnR_n^2/\lambda$ for a suitable $K>0$, then $\{\zeta_j\}_1^n\subset \{\delta(\zeta)<r\}$ with probability at least $1-\epsilon$, or in other words
$$\bfP_n^\beta(\{D_n\ge r\})\le\epsilon.$$
This is proven whenever $\epsilon\gtrsim nR_n^2e^{-\beta k_0 n}$. Taking $\epsilon= e^{-\beta kn}$ where $k=k_0/2$, we
finish the proof. Q.E.D.

\section{Concluding remarks} \label{concrem}

In this section, shall generalize our main results, by allowing for perturbations of the form $u/n$ where
$u$ is a suitable function.
After that, we will comment on related results and say something about future prospects.

\subsection{Perturbations of real-analytic potentials} \label{pertpot} Let us fix a potential $Q$ obeying the
conditions \ref{A0}-\ref{E}.

Now pick an arbitrary bounded, measurable, real-valued function $u$
and consider the $n$-dependent potential
$$V_n(\zeta)=Q(\zeta)+\frac 1 n u(\zeta).$$

In a ``classical'' meaning, the potentials $Q$ and $V_n$ are indistinguishable; their droplets and equilibrium measures are the same.
The difference between them appears
on the statistical level, when we introduce the Gibbs measure corresponding to $V_n$,
$$d\bfP_n^\beta\propto e^{-\beta H_n}\, dA_n,\qquad H_n:=\sum_{j\ne k}^n\log\frac 1 {|\zeta_j-\zeta_k|}+n\sum_{j=1}^n V_n(\zeta_j).$$

More precisely, the weakly $n$-dependent term ($u/n$) affects the distribution of particles near the boundary.

The present more general situation can be treated similarly as before, by redefining the class $\calW_n$ of weighted polynomials
to consist of elements of the form
$$f=q\cdot e^{-nV_n/2},$$
where $q$ is a holomorphic polynomial of degree at most $n-1$. Again we consider this as a subspace of $L^2$.

To an element $f=q\cdot e^{-nV_n/2}\in\calW_n$ we associate the weighted polynomial $\tilde{f}=q\cdot e^{-nQ/2}=f\cdot e^{u/2}$.
Since $u$ is in $L^\infty(\C)$ we have
\begin{equation}\label{asymp}|f(\zeta)|\asymp |\tilde{f}(\zeta)|\end{equation}
uniformly in $f$ and in $\zeta$.

\begin{mth} \label{TH2} Under the above hypotheses, Theorem \ref{DTH} holds up to a constant, i.e., the 1-point function of the ensemble in
potential $V_n$ satisfies $\bfR_n^\beta\lesssim n^2e^{\, s\beta}e^{-cn\beta\delta(\zeta)^2}$. Moreover, Theorem \ref{TH1} and Theorem \ref{mprop} remain in force
as stated (but with $u$-dependent $O$-constants).
\end{mth}

\begin{proof}[Remark on the proof]
In view of \eqref{asymp} we have for each measurable subset $W\subset\C$ and each $j$, $1\le j\le n$, that $Y_{j,W}\asymp \tilde{Y}_{j,W}$ where
$$Y_{j,W}=\int_W |\ell_j|^{2\beta}\, dA\qquad \text{and}\qquad
\tilde{Y}_{j,W}=\int_W |\tilde{\ell}_j|^{2\beta}\, dA.$$ The identity $\bfE_n^\beta(\1_U(\zeta_j)\cdot Y_{j,W})=|U|\bfP_n^\beta(\{\zeta_1\in W\})$ holds as before,
and likewise $\tilde{\bfE}_n^\beta(\1_U(\zeta_j)\cdot \tilde{Y}_{j,W})=|U|\cdot \tilde{\bfP}_n^\beta(\{\zeta_1\in W\})$, where tildes are used to indicate the unperturbed distribution
(with respect to potential $Q$).

Using this, it is easily seen that we lose at most a constant factor when adding the perturbation $u/n$. In Theorem \ref{TH1} the logarithm of that factor can be absorbed in the constant $\mu$, and likewise our argument for Theorem \ref{mprop} goes through essentially unaltered.
\end{proof}

\begin{rem} The assumptions \ref{C} and \ref{E} on the underlying potential $Q$ are automatic from the others if we assume that $Q$ is real-analytic where it is finite.
This follows from Sakai's regularity theorem, see \cite{AKMW}. The perturbation $u/n$ may however be smooth.
\end{rem}

\begin{rem}
It is interesting to view the above result from the perspective of the self-improving method from the paper \cite{AM}. This method was developed with a partial intention to eventually obtain a rigorous proof
of full plane Gaussian field convergence of linear statistics of a Coulomb gas, but due to some technical challenges
it was only applied when $\beta=1$.
(The influential paper \cite{J} provides a somewhat analogous construction on $\R$, which was also applied to $\beta$-ensembles.)

One of the technical obstacles for extending the proof to cover $\beta$-ensembles involved having a good enough decay of the 1-point function
in the exterior of the droplet, which is a case where Theorem \ref{TH2} could be relevant.

The problem of proving Gaussian field convergence in the planar case was later subject of some attention, when the papers \cite{BBNY2,LS} appeared almost simultaneously, proposing different approaches to its solution. This notwithstanding, given a profound statement there are of course always questions of finding alternative explanations. (For the standard Ginibre ensemble, convergence to the Gaussian field was first proved in \cite{RV}.)
\end{rem}

\subsection{Comparison of Theorem \ref{TH1} to earlier results} \label{comp}
Suppose that $\beta=1$, $Q$ is radially symmetric, and $S$ is a disc of radius $R$. In this case we have recognized $R+D_n$ as, essentially,
the spectral radius of a matrix picked randomly from a certain normal matrix ensemble. The distribution of this spectral radius was worked out by
Rider \cite{R} for the Ginibre ensemble (the potential $Q=|\zeta|^2$)
and Chafa\"{i} and P\'{e}ch\'{e} for more general radially symmetric potentials
\cite{CP}.


For a detailed comparison, we introduce the random variables
$$\omega_n=\sqrt{4n\gamma_nc_0}\left(D_n-\sqrt{\frac {\gamma_n}{4nc_0}}\right),\qquad (c_0=\Lap Q(R)).$$
The results in \cite{CP,R} imply that $\omega_n$ converges in distribution to the standard Gumbel distribution as $n\to\infty$, where
$$\gamma_n=\log(n/2\pi)-\log\log n+\log(R^2c_0).$$

(We may recall here that a random variable $X$ is said to have a standard Gumbel distribution if its distribution function is
$\mathbb{P}(X\le t)=\exp(-\exp(-t)).$)

Natural generalizations of these results for ensembles with various types of ``boundary confinements'' are given in the papers \cite{AKS,S}.

\smallskip

The theorems of Rider and Chafa\"{i}-P\'{e}ch\'{e} can be said give a kind of two-dimensional analogue to the well-known convergence to the Tracy-Widom distribution
for the top eigenvalue in Hermitian random matrix theory \cite{ABD,TW}. Numerical evidence in the recent paper \cite{CF} indicates that some similar kind of law might
hold for the $\beta$-Ginibre ensemble.
In this connection, it is interesting to recall that a very precise asymptotic for the tail of the Tracy-Widom $\beta$-distribution for certain one-dimensional ensembles was worked out by Dumaz and Vir\'{a}g in the paper \cite{DV}.

\smallskip

To further compare with our Theorem \ref{TH1} we observe that as $n\to\infty$
\begin{equation}\label{c}\sqrt{\frac {\gamma_n}{4nc_0}}\sim \frac 1 {2\sqrt{c_0}} \sqrt{\frac {\log n}{n}}.\end{equation}
This is of the same order of magnitude ($\sqrt{\frac {\log n}n}$) as our present bound in Theorem \ref{TH1}. Incidentally, we see that our value for the constant $A$ of
proportionality in \eqref{loca} can be improved by a factor $1/2$ in this case.

\smallskip

The low temperature regime when $\beta\asymp \log n$ was studied in \cite{A}. In such a setting, our present results show that the gas is effectively localized to
a microscopic neighbourhood of $S$, \ie, to a neighbourhood of the form
$$\{\zeta;\, \delta(\zeta)\lesssim n^{-1/2}\}.$$
This is used in \cite{AR} to analyze low temperature Coulomb systems.

\smallskip

Due to limitations of our methods, we do not seem to quite reach up to high temperatures of the magnitude $\beta\asymp 1/n$ here.
This kind of regime is however studied, in a suitably adapted setting, in the recent paper \cite{AB}.


\appendix

\section{Johansson's convergence theorem}
In this appendix, we present relevant details about Johansson's convergence theorem \cite[Theorem 2.1]{J} in the planar setting, with varying $\beta$'s. As we shall
see, the arguments in \cite{J} and \cite[Appendix A]{HM} carry through (with some slight extra care) provided that $\beta$ is not too small.


Given a confining potential $Q$ it is convenient to introduce the kernel
$$L_Q(\zeta,\eta):=\log\frac 1 {|\zeta-\eta|}+\frac 1 2 (Q(\zeta)+Q(\eta)).$$

If $\mu$ is a compactly supported Borel measure on $\C$, the logarithmic $Q$-energy can be written as
\begin{equation} \label{logE}
I_Q[\mu]=\iint L_Q(\zeta,\eta)\, d\mu(\zeta)\, d\mu(\eta)=\iint_{\C^2}\log \frac 1 {|\zeta-\eta|}\, d\mu(\zeta)\, d\mu(\eta)+\mu(Q).\end{equation}

Let $\sigma=\Delta Q\cdot \1_S\, dA$ be the equilibrium measure in external potential $Q$, i.e., the minimizer among compactly supported unit charges of the functional \eqref{logE}.
The minimum value of the weighted energy
\begin{equation}\label{robin}
\gamma(Q):=I_Q[\sigma]
\end{equation}
is called the ``modified Robin constant'' in external potential $Q$.

It is pertinent to recall that the assumption \eqref{gro} implies that there are constants $C\in\R$ and $k>1$ such that
\begin{equation}\label{AS4}
Q(\zeta)+C\ge k\log(1+|\zeta|^2),\qquad \zeta\in \C.
\end{equation}

To each configuration $\{\zeta_j\}_1^n$ we associated the empirical measure $\mu_n=\frac 1 n\sum_{j=1}^n\delta_{\zeta_j}$. This measure has infinite $Q$-energy, and we
use as a substitute the discrete energy
\begin{align}\label{discrete}I^\sharp_{Q}[\mu_n]
&:=\frac 1 {n(n-1)}\sum_{j\ne k}L_Q(\zeta_j,\zeta_k)=\frac 1 {n(n-1)}\sum_{j\ne k}\log\frac 1 {|\zeta_j-\zeta_k|}+\mu_n(Q) .
\end{align}
This is closely connected to the Hamiltonian $\Ham_n$ in \eqref{hami}, namely we have
$$\Ham_n=\sum_{j\ne k}L_Q(\zeta_j,\zeta_k)+\sum_{j=1}^n Q(\zeta_j)=n(n-1)I_Q^\sharp[\mu_n]+\sum_{j=1}^n Q(\zeta_j).$$

In the following, we consider $\{\zeta_j\}_1^n$ as a random sample from the Gibbs measure associated with the potential $Q$ and
denote by $\bfR_n^\beta$ the 1-point function.

The following ``Johansson type'' theorem, which partly generalizes \cite[Theorem 2.1]{J} and \cite[Theorem 2.9]{HM} to a case of $n$-dependent $\beta$'s, is the main result of this appendix.

\begin{thm} \label{joh1} Let $\beta=\beta_n$ be a sequence such that $n\beta_n\to\infty$ as $n\to\infty$. Then for any continuous and bounded function $f$ on $\C$ we have
$$\frac 1 n\int_\C f\cdot \bfR_n^{\beta_n}\, dA\to\sigma(f),\qquad n\to\infty.$$
\end{thm}



To prove this theorem, we
fix a small $\epsilon>0$
and form the event
$$A(n,\epsilon)=\left\{I_{Q}^\sharp(\mu_n)\le \gamma(Q)+\epsilon\right\},$$
where (as always) $\mu_n=\sum_1^n\delta_{\zeta_j}$ is picked randomly with respect to $\bfP_n^{\beta_n}$.



\begin{lem} \label{ldp} Fix $a\ge 0$. Then there is a positive integer $n_0$ depending on $\epsilon$ but not on $a$ such that if $n\ge n_0$, then
$$\bfP_n^\beta(\{\mu_N\not\in A(n,\epsilon+a)\})\le e^{-\beta an(n-1)/2+cn}$$
where $c>0$ is a constant depending only on $Q$.
\end{lem}

Our proof of Lemma \ref{ldp} is in two steps.

We start with the following ``entropy estimate'' for the partition function
$Z_n := Z_n^{\beta_n}(Q)=\int_{\C^n}e^{-\beta_n\Ham_n}\, dA_n.$ (The estimate is found for instance in \cite{HM,J}; cf. also \cite[Lemma 4.1]{CHM} for the corresponding statement in $\R^d$.)

\begin{lem} \label{lem4} There are constants $C_1$ and $C_2$  depending only on $Q$ such that for all $n$
$$\frac 1 {n^2}\log Z_n^{\beta_n}\ge -\beta_n(1-C_1/n)\gamma(Q)+C_2/n.$$
\end{lem}

\begin{proof}  Write
\begin{align*}
Z_n=\int_{\C^n} \exp\Big\{ -\beta  \sum_{j\ne k}L_{Q}(\zeta_j,\zeta_k)- \sum_{j=1}^n (\beta Q(\zeta_j)+\log \fii(\zeta_j)) \Big\} \prod_{j=1}^n \, \varphi(\zeta_j)\, dA(\zeta_j).
\end{align*}
where $\varphi\ge 0$ is any continuous compactly supported function with
$\int_{\C} \varphi \, dA=1$.

By Jensen's inequality,
\begin{align*}
\log Z_n &\ge \int_{\C^n} \Big\{ -\beta  \sum_{j\ne k}L_{Q}(\zeta_j,\zeta_k)- \sum_{j=1}^n (\beta Q+\log \fii)(\zeta_j) \Big\} \prod_{j=1}^n \, \varphi(\zeta_j)\, dA(\zeta_j)
\\
&= -\beta n(n-1)I_{Q}[\fii]- n \int_{\C} ( \beta \, Q+\log \varphi ) \, \varphi\, dA,
\end{align*}
with the understanding that $0\log 0=0$, and where we write $I_{Q}[\fii]$ in place of $I_{Q}[\fii\, dA]$.
This leads to
\begin{align}\label{babbel}
\frac 1 {n^2}\log Z_n
&\ge -\beta I_{Q}[\fii]+\frac \beta n I_{Q}[\fii]-\frac \beta n\int Q\varphi\, dA-\frac 1 n\int\varphi\log\varphi\, dA,
\end{align}

For small $\delta>0$ we let $\chi(\zeta):=\delta^{-2}\1_{D(0,\delta)}(\zeta)$ and define a function $\fii_\delta$ by the convolution
$$\fii_{\delta}(\zeta):=\chi*\sigma(\zeta)=\frac{\sigma(D(\zeta,\delta))}{\delta^2},
$$
where $\sigma=\Delta Q\cdot\1_S\, dA$ is the equilibrium measure.

As $\delta\to 0$ we have that $I_{Q}[\fii_{\delta}]\to I_{Q}[\sigma]=\gamma(Q)$; see \cite[pp. 870-871]{HM} for a careful proof of this.

Thus setting $\fii=\fii_\delta$ in \eqref{babbel} and letting $\delta\to 0$ we obtain
\begin{align*}\frac 1 {n^2}\log Z_n&\ge -\beta(1-n^{-1}) \gamma(Q)
-\frac \beta n\int_{S} Q\Delta Q\, dA-\frac 1 n \int_{S}\Delta Q\, \log \Delta Q\, dA.
\end{align*}
The finiteness of the two integrals appearing here follows from our assumptions on $Q$.
\end{proof}

\begin{proof}[Proof of Lemma \ref{ldp}] If $\mu_n\not\in A(n,\epsilon+a)$, then
\begin{equation}\label{con1}I_n^\sharp[\mu_n]\ge \gamma(Q)+\epsilon+a.\end{equation}

We next note that the assumption \eqref{AS4}
and the elementary inequality $|\zeta-\eta|^2\le (1+|\zeta|^2)(1+|\eta|^2)$ for all $\zeta,\eta\in\C$ imply
$L_{Q}(\zeta,\eta)\ge \frac {c_1} 2(Q(\zeta)+Q(\eta))-c_2,$
where $c_1=1-1/k>0$ and $c_2=C/k.$ This gives the inequality
\begin{equation}\label{con2}I_n^\sharp[\mu_n]\ge\frac {c_1}n\sum_{j=1}^n Q(\zeta_j)-c_2.\end{equation}

Now fix a (small) constant $\theta$, $0<\theta<1$ and take a convex combination of the inequalities \eqref{con1} and \eqref{con2}. We obtain
that for all $\mu_n\not\in A(n,\epsilon+a)$,
\begin{equation}\label{above}
I_n^\sharp[\mu_n]\ge (1-\theta) (\gamma(Q)+\epsilon+a)+\theta(\frac {c_1}n\sum_{j=1}^n Q(\zeta_j)- c_2),
\end{equation}
which implies (by use of the inequality \eqref{AS4}),
\begin{align}\Ham_n
 \label{key}
&\ge   n(n-1)(1-\theta)(\gamma(Q)+\epsilon+a)\\
\nonumber &+k((n-1)\theta c_1+1)\sum_{j=1}^n \log(1+|\zeta_j|^2)-n(\theta c_3n+c_4).
\end{align}

Consequently,
\begin{align*}
\int_{\C^n\setminus A(n,\epsilon+a)}&e^{-\beta \Ham_n}\, dA_n\le e^{-\beta n(n-1)(1-\theta)(\gamma(Q)+\epsilon+a)+\beta n(\theta c_3n+c_4)}\\
&\times \Big[\int_\C(1+|\zeta|^2)^{-k\beta((n-1)\theta c_1+1)}\, dA(\zeta)\Big]^n.
\end{align*}
Since
$\int_\C(1+|\zeta|^2)^{-\alpha}\, dA(\zeta)=\frac 1 {\alpha-1}$ for  $\alpha>1$, the integral in brackets is no larger than $1$ when $n$ is large enough, so
\begin{align*}\bfP_n^\beta(\{\mu_N\not\in A(n,\epsilon+a)\})&\le \frac 1 {Z_n}e^{-\beta n(n-1)(1-\theta)(\gamma(Q)+\epsilon+a)+\beta n(\theta c_3n+c_4)},\qquad n\ge n_0.
\end{align*}

We now observe that Lemma~\ref{lem4}
implies 
$$
Z_n\ge \exp(-n(n-1)\beta(\gamma(Q)+\epsilon)-c_5n).$$
Then for $n \ge n_0$,
$$
\bfP_n^\beta(\{\mu_N\not\in A(n,\epsilon+a)\}) \le e^{ \beta n(n-1)[ \theta(\gamma(Q)+c_3+\epsilon)-(1-\theta)a+o(1)  ]+c_5n   }.
$$
Finally we fix $\theta$ with $0<\theta<1/4$ such that
$\theta(\gamma(Q)+c_3+\epsilon)<a/4.$
Then for $n\ge n_0$
$$
\bfP_n^\beta(\{\mu_N\not\in A(n,\epsilon+a)\}) \le e^{ \beta n(n-1)[ a/4-(1-\theta)a  ] +c_5n  }\le e^{-\beta an(n-1)/2+c_5n}.
$$
\end{proof}

\begin{proof}[Proof of Theorem \ref{joh1}] It now suffices to recapitulate a standard argument, which can be found (for example) in \cite[p. 195]{J} in the linear case.

Fix a small $\epsilon>0$.
For each $n\ge n_0(\epsilon)$ we pick with large probability a configuration $\{\zeta_j\}_1^n$ with $\mu_n\in A(n,2\epsilon)$.
Indeed, by Lemma \ref{ldp} the probability for the complementary event is no larger than $e^{-\epsilon\beta n(n-1)/2+cn}$.

The measures $\mu_n$ are then ``tight'' by \eqref{AS4}, i.e., given any $m>0$ there is an $R>0$
such that $\mu_n(\C\setminus D(0,R))<m$ for all $n$. This means that we can extract weakly convergent subsequences (renamed $\mu_n$) converging weakly
to probability measures $\mu^\epsilon$ on $\C$, in the sense that $\mu_n(f)\to\mu^\epsilon(f)$ for each continuous and bounded function $f$.

Letting $\epsilon=\epsilon_n\downarrow 0$ along a suitable sequence, the measures $\mu^{\epsilon_n}$ converge weakly to a probability measure $\mu$ with $I_Q[\mu]\le \gamma(Q)$. This implies $\mu=\sigma$ by unicity of the equilibrium measure, see \cite[Theorem I.1.3]{ST}.  If the convergence $\eps_n\to 0$ is sufficiently slow that $\epsilon_n\beta_n n\to\infty$ as $n\to\infty$,
this happens with large probability, tending to $1$ as $n\to\infty$.

We have shown that with probability $1+o(1)$, every subsequence of the measures $\mu_n$ has a further subsequence converging weakly to $\sigma$, which shows that the full sequence $\mu_n\to\sigma$ weakly.

In particular, the (uniformly bounded) random variables $\mu_n(f)$ converge to $\sigma(f)$ in probability as $n\to\infty$, where $f$ is a continuous and bounded function.
Taking expectations we obtain $\bfE_n^\beta[\mu_n(f)]\to \sigma(f)$ as $n\to\infty$,
 i.e.,
$\frac1n \int_\C f \cdot \bfR_n^{\beta_n}\, dA=\bfE_n^{\beta_n}[\mu_n(f)]\to \sigma(f)$ as $n\to\infty$.
\end{proof}

\end{document}